\documentclass[a4paper,12pt]{article}
\usepackage[centertags]{amsmath}
\usepackage{amsfonts}
\usepackage{amssymb}
\usepackage{amsthm}
\usepackage{amsmath}
\usepackage{dsfont}
\usepackage{graphicx}
\usepackage{tikz}
\definecolor {processblue}{cmyk}{0.96,0,0,0}
\usepackage{amsfonts,graphicx,amsmath,amssymb,hyperref,color}
\usepackage[english]{babel}
\usepackage{authblk}

\newtheorem{thm}{Theorem}[section]

\newtheorem{lem}[thm]{Lemma}
\newtheorem{remark}{Remark}[section]

\numberwithin{equation}{section}

\begin{document}

\renewcommand{\thefootnote}{*}

\begin{center}
{\Large \bf On zeros of some entire  functions
}
\end{center}

\vskip 2mm \centerline{Bing He  }%and Jiang Zeng$^{2}$}
\begin{center}
{\footnotesize  College of Science, Northwest A\&F University\\
   Yangling 712100, Shaanxi, People's Republic of China\\
{\tt yuhe001@foxmail.com}%,\quad http://math.ecnu.edu.cn/\textasciitilde{jwguo}}
\\[10pt]
%$^2$Universit\'e de Lyon; Universit\'e Lyon 1; Institut Camille
%Jordan, UMR 5208 du CNRS;\\ 43, boulevard du 11 novembre 1918,
%F-69622 Villeurbanne Cedex, France\\
%{\tt zeng@math.univ-lyon1.fr,\quad
%http://math.univ-lyon1.fr/\textasciitilde{zeng}}
}
\end{center}

\vskip 0.7cm \noindent{\small{\bf Abstract.}} Let \begin{equation*}
  A_{q}^{(\alpha)}(a;z)=\sum_{k=0}^{\infty}\frac{(a;q)_{k}q^{\alpha k^2} z^k}{(q;q)_{k}},
\end{equation*}
where $\alpha >0,~0<q<1.$ In a paper of Ruiming Zhang, he asked under what conditions the zeros of the entire
function $A_{q}^{(\alpha)}(a;z)$ are all real and established some results on the zeros of $A_{q}^{(\alpha)}(a;z)$ which present a partial answer to that question. In the present paper, we will set up some results on certain entire  functions which includes that $A_{q}^{(\alpha)}(q^l;z),~l\geq 2$ has only infinitely many negative zeros that gives a partial answer to Zhang's question. In addition, we establish some results on zeros of certain entire  functions involving the Rogers-Szeg\H{o} polynomials and the Stieltjes-Wigert polynomials.
\section{Introduction}
%$\lim\limits_{q\rightarrow 1}[n]=n.$
%$$\mbox{ for } n\geq 1$$
%$$\text{ for }n\geq 1$$
%$${}_{2}F_{1}$$
%$$\sum_{\scriptstyle 1<i<n \atop \scriptstyle 2|i} 1/i$$
%\begin{eqnarray*}
%x&<&y+z\\
% &=&u+v
%\end{eqnarray*}
%\begin{align*}
%\sum_{k=0}^{3^a m-1} q^{k}{2k\brack k}_q
%&\equiv 0 \pmod{(1-q^{3^a})/(1-q) },  \\
%\sum_{k=0}^{5^a m-1}(-1)^kq^{-{k+1\choose 2}}{2k\brack k}_q
%&\equiv 0 \pmod{(1-q^{5^a})/(1-q)},
%\end{align*}
Recall that entire functions are functions that are holomorphic in the whole complex plane. Given an entire function $f(z)=\sum_{k=0}^\infty a_{k}z^k,$ then the order of $f(z)$ can be computed by \cite [(2.2.3)]{B}
\begin{equation}\label{e1-1}
\rho(f)=\limsup_{k\rightarrow \infty}\frac{k\log k}{-\log a_{k}}.
\end{equation}
Following \cite{IZ}, we define the entire function $A_{q}^{(\alpha)}(a;z)$ by
\begin{equation*}
  A_{q}^{(\alpha)}(a;z)=\sum_{k=0}^{\infty}\frac{(a;q)_{k}q^{\alpha k^2} z^k}{(q;q)_{k}},
\end{equation*}
where $\alpha >0,~0<q<1$ and
\begin{equation*}
  (a;q)_{0}=1,~(a;q)_{k}=\prod_{j=0}^{k-1}(1-aq^{j})~(k\geq 1).
\end{equation*}
It is easily seen that
\begin{align*}
   A_{q}^{(\frac{1}{2})}(q^{-n};z)&=\sum_{k=0}^{\infty}\frac{(q^{-n};q)_{k}q^{\frac{k^2}{2}} z^k}{(q;q)_{k}}=(q;q)_{n}S_{n}(zq^{\frac{1}{2}-n};q),\\
   A_{q}^{(1)}(0;z)&=\sum_{k=0}^{\infty}\frac{q^{k^2}z^k}{(q;q)_{k}}=A_{q}(-z),
\end{align*}
where $A_{q}(z)$ and $S_{n}(z;q)$ are the Ramanujan entire function and the Stieltjes-Wigert
polynomial respectively \cite{I}. So $A_{q}^{(\alpha)}(a;z)$ generalizes both $A_{q}(z)$ and $S_{n}(z;q).$ It is well-known that both of them have only real positive
zeros. Therefore,  Zhang in \cite{Z} asked under what conditions the zeros of the entire
function $A_{q}^{(\alpha)}(a;z)$ are all real. In that paper, Zhang proved that $A_{q}^{(\alpha)}(-a;z)~(a\geq 0, \alpha> 0, 0<q<1)$ has only infinitely many negative zeros and
$A_{q}^{(\alpha)}(q^{-n};z)~(n\in \mathbb{N}, \alpha\geq 0, 0<q<1)$  has only finitely many positive zeros, which gave a partial answer to that question. In addition, Zhang obtained
a result on the negativity of zeros of an entire function including many well-known entire functions.

Our motivation for the present work emanates from Zhang's question. In this paper, we will establish the following results which present a partial answer to Zhang's question.
\begin{thm}\label{t1}Let $\alpha > 0$ and $0<q<1.$ Then

\emph{(i)}if $l\geq 2$ is an integer, then $A_{q}^{(\alpha)}(q^l;z)$ has only infinitely many real zeros and all of them are negative;

\emph{(ii)}if $m$ and $n$ are nonnegative integers such that at least one of them is positive,~$\{l_{j}\}_{j=1}^{m}$ are integers not less than $2,~0<q_{j}<1~(1\leq j\leq m)$ and $\nu_{r}>-1,~0<q_{r}<1~(1\leq r\leq n),$ then the function
\begin{equation*}
\sum_{k=0}^{\infty}\prod_{j=1}^{m}\frac{(q_{j}^{l_{j}};q_{j})_{k}}{(q_{j};q_{j})_{k}}\frac{q^{\alpha k^2}}{\prod_{r=1}^n (q_{r},q_{r}^{\nu_{r}+1};q_{r})_{k}}z^k
\end{equation*}
has only infinitely many real zeros and all of them are negative;

\emph{(iii)}if $m\geq 0$  and $n\geq 1$ are integers,~$\{l_{j}\}_{j=1}^{m}$ are integers not less than $2$ and $\nu_{r}\geq0 ~(1\leq r\leq m),$ then the function
\begin{equation*}
\sum_{k=0}^{\infty}\frac{\prod_{j=1}^{m}(l_{j})_{k}}{(k!)^{m+n}\prod_{r=1}^n (\nu_{r}+1)_{k}}z^k
\end{equation*}
where $(a)_{k}$ is defined by $(a)_{0}=1,~(a)_{k}=a(a+1)\cdots(a+k-1)~(k\geq 1),$
has only infinitely many real zeros and all of them are negative.
\end{thm}
It should be mentioned that in \cite [Theorem 4]{KLV} Katkova
et al. proved that there exists a constant $q_{\infty}~(\approx 0.556415)$ such that the function $A_{q}^{(\alpha)}(q;z)$ has only real zeros if and only if $q\leq q_{\infty}.$ So the similar result for $A_{q}^{(\alpha)}(q^l;z)$ does not hold for  $l=1.$

The Gaussian binomial coefficients are $q$-analogs of the binomial coefficients, which are given by
$${n\brack k}_{q}=\frac{(q;q)_{n}}{(q;q)_{k}(q;q)_{n-k}}.$$
We now introduce the definition of the Rogers-Szeg\H{o} polynomials which
were first investigated by Rogers \cite{R}  and then by Szeg\H{o} \cite{S}.  The Rogers-Szeg\H{o} polynomials are defined by
\begin{equation*}
   h_{n}(x,y|q)=\sum_{k=0}^n{n\brack k}_{q}x^k y^{n-k}.
\end{equation*}
If $q$ is replaced by $q^{-1}$ in the Rogers-Szeg\H{o} polynomials, then we obtain
the Stieltjes-Wigert polynomials (see \cite{S}):
\begin{equation*}
    g_{n}(x,y|q)=\sum_{k=0}^n{n\brack k}_{q}q^{k(k-n)}x^k y^{n-k}.
\end{equation*}
From \cite [Theorem 5]{Z}, we know that $ h_{n}(x|q)$ has only negative zeros for $q\geq 1$ and $g_{n}(x|,q)$ has only negative zeros for $0<q\leq 1,$ where $h_{n}(x|q)$ and
$g_{n}(x|,q)$ are defined by
\begin{equation*}
   h_{n}(x|q):=h_{n}(x,1|q)=\sum_{k=0}^n{n\brack k}_{q}x^k
\end{equation*}
and
\begin{equation*}
    g_{n}(x|q):= g_{n}(x,1|q)=\sum_{k=0}^n{n\brack k}_{q}q^{k(k-n)}x^k.
\end{equation*}

Motivated by Zhang's work, we will establish the following results on zeros of certain entire functions involving the Rogers-Szeg\H{o} polynomials and the Stieltjes-Wigert polynomials.
\begin{thm}\label{tt1}Let $0<q<1.$  If  $\alpha$ is positive number and $0<x,y<1,$  then
$$\sum_{n=0}^{\infty}\frac{h_{n}(x,y|q)}{(q;q)_{n}}q^{\alpha n^2}z^n$$
has infinitely many real zeros and all of them are negative; if $-1<x,y<0$ and $\alpha\geq \frac{1}{2},$ then
\begin{equation*}
  \sum_{n=0}^{\infty}\frac{g_{n}(x,y|q)}{(q;q)_{n}}q^{\alpha n^2}z^n
\end{equation*}
has infinitely many real zeros and all of them are positive.
\end{thm}
\begin{remark}
\emph{(i)}~Applying the method which is used in the proof of Theorem \ref{tt1}, we can deduce the following results: let $0<q<1;$
if  $\alpha$ is positive number and $0<x<1,$  then
$$\sum_{n=0}^{\infty}\frac{h_{n}(x|q)}{(q;q)_{n}}q^{\alpha n^2}z^n$$
has infinitely many real zeros and all of them are negative; if $-1<x<0$ and $\alpha\geq \frac{1}{2},$ then
\begin{equation*}
  \sum_{n=0}^{\infty}\frac{g^{-}_{n}(x|q)}{(q;q)_{n}}q^{\alpha n^2}z^n
\end{equation*}
has infinitely many real zeros and all of them are positive, where $g^{-}_{n}(x|q)=g_{n}(x,-1|q).$ But we need the following results:
\begin{align*}
  \left|\frac{h_{n}(x|q)}{(q;q)_{n}}\right|&\leq \frac{1}{(q,x;q)_{\infty}},~
  \left|\frac{g^{-}_{n}(x|q)}{(q;q)_{n}}\right|\leq \frac{1}{(q,|x|;q)_{\infty}}
\end{align*}
which can be derived easily.

\emph{(ii)} We can establish certain results on the Rogers-Szeg\H{o} polynomials and the Stieltjes-Wigert polynomials by using similar method. These are analogous to (ii) and (iii) of Theorem \ref{t1}
\end{remark}

We also set up the following result which is analogous to \cite [Theorem 7]{Z}.
\begin{thm}\label{t2}
Suppose  $r$ and $s$ are two positive integers, $a_{j}(1\leq j\leq r)$ and $b_{k}(1\leq k\leq s)$ are r+s positive numbers and $\alpha> 0, ~0<q<2^{-\frac{1}{\alpha}}.$ Then there exists $K_{0}\in \mathbb{Z}_{> 0}$ such that for all integers $K\geq K_{0},$ the function
\begin{equation*}
  \sum_{k=K}^{\infty}\frac{(a_{1})_{k}(a_{2})_{k}\cdots (a_{r})_{k}}{(b_{1})_{k}(b_{2})_{k}\cdots (b_{s})_{k}}q^{\alpha k^2}z^k
\end{equation*}
has only infinitely many real zeros and all of them are negative.
\end{thm}

In the next section, we will provide some lemmas which are crucial in the proof of Theorems \ref{t1} and \ref{tt1}. Section 3 is devoted to our proof of Theorems \ref{t1}--\ref{t2}.
\section{Preliminaries}
In order to prove Theorems \ref{t1} and \ref{tt1}, we need some auxiliary results. We first recall from \cite{DP} that a real entire function $f(z)$ is of Laguerre-P\'{o}lya class if
\begin{equation*}
  f(z)=cz^m e^{-\alpha z^2+\beta z}\prod_{k=1}^{\infty}\left(1+\frac{z}{z_{k}}\right)e^{-z/z_{k}},
\end{equation*}
where $c, \beta, z_{k}\in \mathbb{R}, \alpha \geq 0, m\in \mathbb{Z}_{\geq 0}$ and $\sum_{k=1}^{\infty}z_{k}^{-2}< +\infty.$

Let us recall that a real sequence $\{a_{n}\}_{n=0}^{\infty}$ is called a P\'{o}lya frequence (or PF) sequence if the infinite
matrix $(a_{j-i})_{i,j=0}^{\infty}$ is totally positive, i.e. all its minors are nonnegative, where we use the notation $a_{k}=0$ if $k<0.$ This concept can be extended to finite
sequences in the obvious way by completing the sequence with zero terms.
\begin{lem}\emph{(See \cite{AESW})}\label{l1}The sequence $\{a_{k}\}_{k=0}^{\infty}$ is a PF sequence if and only if the convergent series $\sum_{k=0}^{\infty}a_{k}z^k$ satisfies
\begin{equation*}
  \sum_{k=0}^{\infty}a_{k}z^k=cz^me^{\gamma z}\prod_{k=1}^{\infty}\frac{1+\alpha_{k}z}{1-\beta_{k}z},
\end{equation*}
where $c\geq 0, \gamma\geq 0, \alpha_{k}\geq 0, \beta_{k}\geq 0, m\in \mathbb{Z}^{+}$ and $\sum_{k=1}^{\infty}(\alpha_{k}+\beta_{k})<+\infty.$
\end{lem}
It was proved in \cite{CPP} that
\begin{equation*}
\sum_{k=0}^{\infty}\frac{q^{k^2}}{k!}x^k
\end{equation*}
is a real entire function and in Laguerre-P\'{o}lya class for $|q|< 1$ and $x\in \mathbb{R}.$ Then by \cite [Theorem C]{DP}, we obtain that $\{\frac{q^{k^2}}{k!}\}_{k=0}^{\infty}$ is a PF sequence for $0<q<1.$
\begin{lem}\emph{(See \cite [p. 1047]{DJM} and \cite{PS})}\label{l2}
Let $\{a_{k}\}_{k=0}^{m}$ and $\{b_{k}\}_{k=0}^{n}$ be  sequences of nonnegative numbers. Then

 \emph{(i)}~the sequence $\{a_{k}\}_{k=0}^{m}$ is a a PF sequence if and only if the polynomial $\sum_{k=0}^{m}a_{k}z^k$ has only nonpositive zeros;

 \emph{(ii)}~if the sequences $\{a_{k}\}_{k=0}^{m}$ and $\{b_{k}\}_{k=0}^{n}$ are PF sequences, then so is the sequence $\{a_{k}\cdot b_{k}\}_{k=0}^{\infty};$

 \emph{(iii)}~if the sequences $\{a_{k}\}_{k=0}^{m}$ and $\{b_{k}\}_{k=0}^{n}$ are PF sequences, then so is the sequence $\{k!\cdot a_{k}\cdot b_{k}\}_{k=0}^{\infty}.$
\end{lem}
We also need the following results, namely, Vitali's theorem \cite{T} and Hurwitz's theorm \cite [\S 5, Theorem 2]{Ah}.
\begin{lem}\emph{(Vitali's theorem)}\label{l3}
Let $\{f_{n}(z)\}$ be a sequence of functions analytic in a domain $D$ and assume that $f_{n}(z)\rightarrow f(z)$  point-wise in $D.$ Then $f_{n}(z)\rightarrow f(z)$ uniformly in any subdomain bounded by a contour $C,$ provided that $C$ is contained in $D.$
\end{lem}
\begin{lem}\emph{(Hurwitz's theorm)}\label{l4}
If the functions $\{f_{n}(z)\}$ are nonzero and analytic  in a region $\Omega,$  and $f_{n}(z)\rightarrow f(z)$ uniformly on every compact subset of $\Omega,$ then $f(z)$ either identically zero or never equal to zero in $\Omega.$
\end{lem}
We conclude this section with following result which is  very important in the proof of Theorem \ref{tt1}.
\begin{lem}\label{l5} Let $x,y$ be two real numbers and $\alpha>0,~ 0<q<1.$  Then the functions
\begin{equation*}
\sum_{n=0}^{\infty}\frac{h_{n}(x,y|q)}{(q;q)_{n}}q^{\alpha n^2}z^n~\text{and}~\sum_{n=0}^{\infty}\frac{g_{n}(x,y|q)}{(q;q)_{n}}q^{\alpha n^2}z^n
\end{equation*}
are all entire functions.
\end{lem}
\begin{proof}
We first consider the function $\sum_{n=0}^{\infty}\frac{h_{n}(x,y|q)}{(q;q)_{n}}q^{\alpha n^2}z^n.$
For $|x|\leq 1, |y|\leq 1,$ by \cite [(1.3.15)]{GR}, we have
\begin{align*}
   |h_{n}(x,y|q)|&\leq\sum_{k=0}^n{n\brack k}_{q}|x|^k|y|^{n-k}\leq \sum_{k=0}^n{n\brack k}_{q}\\
   &\leq \sum_{k=0}^n \frac{(1-q^n)(1-q^{n-1})\cdots (1-q^{n-k+1})}{(q;q)_{k}}q^{k-n}\\
   &\leq q^{-n}\sum_{k=0}^n\frac{q^k}{(q;q)_{k}}\leq q^{-n}\sum_{k=0}^\infty\frac{q^k}{(q;q)_{k}}\\
   &=\frac{q^{-n}}{(q;q)_{\infty}}.
\end{align*}
By the same arguments, we get for $|x|> 1, |y|\leq 1,$
\begin{align*}
 |h_{n}(x,y|q)|&\leq \sum_{k=0}^n{n\brack k}_{q}|y|^k|x|^{n-k}= |x|^n\sum_{k=0}^n{n\brack k}_{q}|y|^k|x|^{-k}\\
 &\leq |x|^n\sum_{k=0}^n{n\brack k}_{q}\leq\frac{(|x|/q)^{n}}{(q;q)_{\infty}};
\end{align*}
for $|x|\leq1, |y|> 1,$
\begin{align*}
  |h_{n}(x,y|q)|&\leq |y|^n\sum_{k=0}^n{n\brack k}_{q}|x|^k|y|^{-k}\\
  &\leq |y|^n\sum_{k=0}^n{n\brack k}_{q}\leq\frac{(|y|/q)^{n}}{(q;q)_{\infty}};
\end{align*}
for $|x|>1, |y|> 1,$
\begin{align*}
  |h_{n}(x,y|q)|&=|xy|^n\sum_{k=0}^n{n\brack k}_{q}|x|^{k-n}|y|^{-k}\\
  &\leq |xy|^n\sum_{k=0}^n{n\brack k}_{q}\leq\frac{(|xy|/q)^{n}}{(q;q)_{\infty}}.
\end{align*}
In any cases we obtain
\begin{equation*}
  |h_{n}(x,y|q)|\leq\frac{a^{n}}{(q;q)_{\infty}}
\end{equation*}
where $a$ is positive number which depends on $x,y$ and $q.$  Then
\begin{equation*}
\left|\frac{h_{n}(x,y|q)}{(q;q)_{n}}q^{\alpha n^2}\right|\leq\frac{a^{n}q^{\alpha n^2}}{(q;q)^2_{\infty}}.
\end{equation*}
so that
\begin{equation*}
\limsup_{n\rightarrow \infty}\left|\frac{h_{n}(x,y|q)}{(q;q)_{n}}q^{\alpha n^2}\right|^{\frac{1}{n}}=0,
\end{equation*}
which proves that
\begin{equation*}
  \sum_{n=0}^{\infty}\frac{h_{n}(x,y|q)}{(q;q)_{n}}q^{\alpha n^2}z^n
\end{equation*}
is an entire function. Similarly, we can deduce that
\begin{equation*}
  \sum_{n=0}^{\infty}\frac{g_{n}(x,y|q)}{(q;q)_{n}}q^{\alpha n^2}z^n
\end{equation*}
is also an entire function.
This ends the proof of Lemma \ref{l5}.
\end{proof}
\section{Proof of Theorems \ref{t1}--\ref{t2}}
\emph{Proof of Theorem \ref{t1}}.
We first prove (i).
According to the $q$-binomial theorem \cite{An, GR}, we obtain that  for all complex numbers  $x$ and $q$ with $|x|<1$ and $|q|<1,$ there holds
\begin{equation}\label{e3-1}
  \sum_{k=0}^{\infty}\frac{(a;q)_{k}}{(q;q)_{k}}x^k=\frac{(ax;q)_{\infty}}{(x;q)_{\infty}}.
\end{equation}
Setting $a=q^l, x=zq^{-l}$ in \eqref{e3-1} gives
\begin{equation*}
  \sum_{k=0}^{\infty}\frac{(q^l;q)_{k}}{(q;q)_{k}}q^{-lk}z^k=\frac{(z;q)_{\infty}}{(zq^{-l};q)_{\infty}}=\frac{1}{(zq^{-l};q)_{l}}.
\end{equation*}
Using Lemma \ref{l1}, we get the sequence
\begin{equation*}
  \left\{\frac{(q^l;q)_{k}}{(q;q)_{k}}q^{-lk}\right\}_{k=0}^{n}
\end{equation*}
is a PF sequence. It follows from (i) of Lemma \ref{l2} that $\{a_{k}\}_{k=0}^{n}$ is a a PF sequence if and only if $\{c^k\cdot a_{k}\}_{k=0}^{n}$ is also a PF sequence for any $c> 0.$
Hence, \begin{equation*}
  \left\{\frac{(q^l;q)_{k}}{(q;q)_{k}}\right\}_{k=0}^{n}
\end{equation*}
is a PF sequence. So from the fact $\left\{\frac{q^{\alpha k^2}}{k!}\right\}_{k=0}^{n}$ is a PF sequence and (iii) of Lemma \ref{l2}, we arrive at the sequence
\begin{equation*}
  \left\{\frac{(q^l;q)_{k}q^{\alpha k^2}}{(q;q)_{k}}\right\}_{k=0}^{n}
\end{equation*}
is also a PF sequence, which, by (i) of Lemma \ref{l2}, implies that the polynomial
\begin{equation*}
  \sum_{k=0}^{n}\frac{(q^l;q)_{k}q^{\alpha k^2}}{(q;q)_{k}}z^k
\end{equation*}
has only nonpositive zeros. Here and below, set $\Omega=\mathbb{C}-\{x+yi|x\in (-\infty, 0], y=0\}.$ Then $$\sum_{k=0}^{n}\frac{(q^l;q)_{k}q^{\alpha k^2}}{(q;q)_{k}}z^k\rightarrow A_{q}^{\alpha}(q^l;z)$$
point-wise in $\Omega.$ It is easily seen that for $0<q<1,~\alpha\geq 0$ and each $n\in \mathbb{N},~z\in \mathbb{C},$
\begin{equation*}
  \bigg|\sum_{k=0}^{n}\frac{(q^l;q)_{k}q^{\alpha k^2}}{(q;q)_{k}}z^k\bigg|\leq \sum_{k=0}^{\infty}\frac{(q^l;q)_{k}q^{\alpha k^2}}{(q;q)_{k}}|z|^k<+\infty.
\end{equation*}
We apply Lemma \ref{l3} to know that $$\sum_{k=0}^{n}\frac{(q^l;q)_{k}q^{\alpha k^2}}{(q;q)_{k}}z^k\rightarrow A_{q}^{\alpha}(q^l;z)$$
uniformly on every compact subset of $\Omega$ and then apply Lemma \ref{l4} to see that $A_{q}^{\alpha}(q^l;z)\neq 0$ in $\Omega$
which means that $A_{q}^{\alpha}(q^l;z)$ has no zeros outside the set $\{x+yi|x\in (-\infty, 0], y=0\}.$  According to \cite [Lemma 14.1.4]{I},
we have $A_{q}^{\alpha}(q^l;z)$ has  infinitely many zeros. Therefore, $A_{q}^{\alpha}(q^l;z)$ has only infinitely many real zeros and all of them are negative, which proves (i).

We next show (ii). According to \cite{Z}, we know that the sequence $$\left\{\frac{1}{(q,q^{\nu+1};q)_{k}}\right\}_{k=0}^{N}$$
is a PF sequence for $\nu>-1,~0<q<1,$ which means that
\begin{equation*}
  \left\{\frac{1}{(q_{r},q_{r}^{\nu_{r}+1};q_{r})_{k}}\right\}_{k=0}^{N}
\end{equation*}
are all PF sequences for $1\leq r\leq n.$

Since \begin{equation*}
  \left\{\frac{(q_{j}^{l_{j}};q_{j})_{k}}{(q_{j};q_{j})_{k}}\right\}_{k=0}^{N}~(1\leq j\leq m),~\left\{\frac{1}{(q_{r},q_{r}^{\nu_{r}+1};q_{r})_{k}}\right\}_{k=0}^{N}~(1\leq r\leq n)
\end{equation*}
and $\left\{\frac{q^{\alpha k^2}}{k!}\right\}_{k=0}^{n}$ are all PF sequences, we then apply (ii) and (iii) of Lemma \ref{l2} to find that
\begin{equation*}
  \left\{\prod_{j=1}^{m}\frac{(q_{j}^{l_{j}};q_{j})_{k}}{(q_{j};q_{j})_{k}}\frac{q^{\alpha k^2}}{\prod_{r=1}^n (q_{r},q_{r}^{\nu_{r}+1};q_{r})_{k}}\right\}_{k=0}^{N}
\end{equation*}
is also a PF sequence, which, by (i) of Lemma \ref{l2}, implies that
\begin{equation*}
\sum_{k=0}^{N}\prod_{j=1}^{m}\frac{(q_{j}^{l_{j}};q_{j})_{k}}{(q_{j};q_{j})_{k}}\frac{q^{\alpha k^2}z^k}{\prod_{r=1}^n (q_{r},q_{r}^{\nu_{r}+1};q_{r})_{k}}
\end{equation*}
has only negative zeros. For each positive integer $N$ and $z\in \mathbb{C},$  we have
\begin{align*}
  &\left|\sum_{k=0}^{N}\prod_{j=1}^{m}\frac{(q_{j}^{l_{j}};q_{j})_{k}}{(q_{j};q_{j})_{k}}\frac{q^{\alpha k^2}z^k}{\prod_{r=1}^n (q_{r},q_{r}^{\nu_{r}+1};q_{r})_{k}}\right|\\
  &\leq \sum_{k=0}^{\infty}\prod_{j=1}^{m}\frac{(q_{j}^{l_{j}};q_{j})_{k}}{(q_{j};q_{j})_{k}}\frac{q^{\alpha k^2}|z|^k}{\prod_{r=1}^n (q_{r},q_{r}^{\nu_{r}+1};q_{r})_{k}}<+\infty.
\end{align*}
Similarly, we apply Lemmas \ref{l3} and \ref{l4} to establish that the function
\begin{equation*}
\sum_{k=0}^{\infty}\prod_{j=1}^{m}\frac{(q_{j}^{l_{j}};q_{j})_{k}}{(q_{j};q_{j})_{k}}\frac{q^{\alpha k^2}z^k}{\prod_{r=1}^n (q_{r},q_{r}^{\nu_{r}+1};q_{r})_{k}}
\end{equation*}
has no zeros outside the set $\{x+yi|x\in (-\infty, 0], y=0\}.$  In view of \cite [Lemma 14.1.4]{I},
this function has  infinitely many zeros. Then (ii) is proved.

Finally, we give a proof of (iii). Let $q_{j}=q_{r}=q$ and $\alpha =n+m/2.$  Using the H\^{o}pital's rule, we deduce that
\begin{equation*}
  \lim_{q\rightarrow 1}\frac{(q^{l_{j}};q)_{k}}{(q;q)_{k}}=\frac{(l_{j})_{k}}{k!},~\lim_{q\rightarrow 1}\frac{(1-q)^{2k}}{(q,q^{\nu_{r}+1};q)_{k}}=\frac{1}{k!(\nu_{r}+1)_{k}}.
\end{equation*}
Then
\begin{equation*}
\lim_{q\rightarrow 1}\frac{(1-q)^{2kn}\prod_{j=1}^{m}(q^{l_{j}};q)_{k}q^{(n+m/2) k^2}\cdot z^k}{(q;q)_{k}^{m+n}\prod_{r=1}^n (q^{\nu_{r}+1};q)_{k}}=\frac{\prod_{j=1}^{m}(l_{j})_{k}\cdot z^k}{(k!)^{m+n}\prod_{r=1}^n (\nu_{r}+1)_{k}}.
\end{equation*}
It is easy to see from
\begin{equation*}
lq^{l-1}\leq \frac{1-q^l}{1-q}=1+q+q^2+\cdots+q^{l-1}\leq l
\end{equation*}
that
\begin{equation*}
  l!q^{{l\choose 2}}\leq \frac{(q;q)_{l}}{(1-q)^l}\leq l!.
\end{equation*}
Then
\begin{equation*}
\frac{1}{l!}\leq \frac{(1-q)^l}{(q;q)_{l}}\leq \frac{q^{-{l\choose 2}}}{l!}\leq \frac{q^{-l^2/2}}{l!}.
\end{equation*}
Combining this and the fact that $\frac{(1-q)^l}{(q^b;q)_{l}}\leq \frac{(1-q)^l}{(q;q)_{l}}$ for $b\geq 1$ gives
\begin{equation*}
  \left|\frac{(1-q)^{2kn}\prod_{j=1}^{m}(q^{l_{j}};q)_{k}q^{\alpha k^2}\cdot z^k}{(q;q)_{k}^{m+n}\prod_{r=1}^n (q^{\nu_{r}+1};q)_{k}}\right|\leq \frac{|z|^k}{k!^{m+2n}}
\end{equation*}
This, by Lemma \ref{l3},  shows that
\begin{equation*}
\lim_{q\rightarrow 1}\sum_{k=0}^{\infty}\frac{(1-q)^{2kn}\prod_{j=1}^{m}(q^{l_{j}};q)_{k}q^{(n+m/2) k^2}\cdot z^k}{(q;q)_{k}^{m+n}\prod_{r=1}^n (q^{\nu_{r}+1};q)_{k}}=\sum_{k=0}^{\infty}\frac{\prod_{j=1}^{m}(l_{j})_{k}\cdot z^k}{(k!)^{m+n}\prod_{r=1}^n (\nu_{r}+1)_{k}}.
\end{equation*}
converges uniformly in in any compact subset of $\mathbb{C}.$   It follows from Lemma \ref{l4} that the function
\begin{equation*}
  \sum_{k=0}^{\infty}\frac{\prod_{j=1}^{m}(l_{j})_{k}\cdot z^k}{(k!)^{m+n}\prod_{r=1}^n (\nu_{r}+1)_{k}}
\end{equation*}
has no zeros outside the set $\{x+yi|x\in (-\infty, 0], y=0\}.$

Set $$a_{k}=\frac{\prod_{j=1}^{m}(l_{j})_{k}}{(k!)^{m+n}\prod_{r=1}^n (\nu_{r}+1)_{k}}.$$ It is easily seen from the Stirling's formula \cite{AAR} that
\begin{equation*}
\lim_{k\rightarrow \infty}\frac{-\log a_{k}}{k\log k}=2n
\end{equation*}
which, by \eqref{e1-1}, means that
\begin{equation*}
  \rho\left(\sum_{k=0}^{\infty}\frac{\prod_{j=1}^{m}(l_{j})_{k}\cdot z^k}{(k!)^{m+n}\prod_{r=1}^n (\nu_{r}+1)_{k}}\right)=\frac{1}{2n}\leq \frac{1}{2}.
\end{equation*}
Hence, by \cite [Theorem 1.2.5]{I}, the function \begin{equation*}
  \sum_{k=0}^{\infty}\frac{\prod_{j=1}^{m}(l_{j})_{k}\cdot z^k}{(k!)^{m+n}\prod_{r=1}^n (\nu_{r}+1)_{k}}
\end{equation*} has infinitely many zeros. Then this function has only infinitely many real zeros and all of them are negative, which proves (iii).
This completes the proof of Theorem \ref{t1}.\qed\\
\emph{Proof of Theorem \ref{tt1}}. We first consider the function $\sum_{n=0}^{\infty}\frac{h_{n}(x,y|q)}{(q;q)_{n}}q^{\alpha n^2}z^n,$ where $0<x,y<1$ and $\alpha>0.$ From \cite [Theorem 3.1, (3.1)]{L}, we know that
\begin{equation*}
\sum_{n=0}^{\infty}\frac{h_{n}(x,y|q)}{(q;q)_{n}}t^n=\frac{1}{(xt,yt;q)_{\infty}}
\end{equation*}
for $\max\{|xt|,|yt|\}<1.$ Then, by Lemma \ref{l1}, we have the sequence
\begin{equation*}
\left\{\frac{h_{n}(x,y|q)}{(q;q)_{n}}\right\}_{n=0}^{\infty}
\end{equation*}
is a PF sequence.  It follows from the fact that $\{\frac{q^{k^2}}{k!}\}_{k=0}^{\infty}$ is a PF sequence and (iii) in Lemma \ref{l2} that
\begin{equation*}
  \left\{\frac{h_{n}(x,y|q)}{(q;q)_{n}}q^{\alpha n^2}\right\}_{n=0}^{N}
\end{equation*}
is also a PF sequence. So, by (i) of Lemma \ref{l2}, we see that
\begin{equation*}
  \sum_{n=0}^{N}\frac{h_{n}(x,y|q)}{(q;q)_{n}}q^{\alpha n^2}z^n
\end{equation*}
has only nonpositive zeros. We know that
\begin{equation*}
 \sum_{n=0}^{N}\frac{h_{n}(x,y|q)}{(q;q)_{n}}q^{\alpha n^2}z^n \rightarrow \sum_{n=0}^{\infty}\frac{h_{n}(x,y|q)}{(q;q)_{n}}q^{\alpha n^2}z^n
\end{equation*}
point-wise in $\Omega.$  It is easy to see that for $0<q<1,~\alpha> 0$ and each $N\in \mathbb{N},~z\in \mathbb{C},$
\begin{equation*}
\left|\sum_{n=0}^{\infty}\frac{h_{n}(x,y|q)}{(q;q)_{n}}q^{\alpha n^2}z^n\right|\leq \sum_{n=0}^{\infty}\frac{h_{n}(x,y|q)}{(q;q)_{n}}q^{\alpha n^2}|z|^n<+\infty.
\end{equation*}
Applying Lemma \ref{l3}, we find that
\begin{equation*}
 \sum_{n=0}^{N}\frac{h_{n}(x,y|q)}{(q;q)_{n}}q^{\alpha n^2}z^n \rightarrow \sum_{n=0}^{\infty}\frac{h_{n}(x,y|q)}{(q;q)_{n}}q^{\alpha n^2}z^n
\end{equation*}
uniformly on every compact subset of $\Omega$ and then applying Lemma \ref{l4}, we  deduce that the function
$$\sum_{n=0}^{\infty}\frac{h_{n}(x,y|q)}{(q;q)_{n}}q^{\alpha n^2}z^n\neq 0$$ in $\Omega$  which means that
$$\sum_{n=0}^{\infty}\frac{h_{n}(x,y|q)}{(q;q)_{n}}q^{\alpha n^2}z^n$$
has no zeros outside the set $\{x+yi|x\in (-\infty, 0], y=0\}.$

We use  \cite [(1.3.15)]{GR} to get
\begin{align*}
    0\leq\frac{h_{n}(x,y|q)}{(q;q)_{n}}&=\sum_{k=0}^{n}\frac{x^k}{(q;q)_{k}}\frac{y^{n-k}}{(q;q)_{n-k}}\\
    &\leq \sum_{k=0}^{\infty}\frac{x^k}{(q;q)_{k}}\sum_{k=0}^{\infty}\frac{y^k}{(q;q)_{k}}\\
    &=\frac{1}{(x,y;q)_{\infty}}.
\end{align*}
Then, by Lemma \ref{l5} and \cite [Lemma 14.1.4]{I}, we attain that the function
$$\sum_{n=0}^{\infty}\frac{h_{n}(x,y|q)}{(q;q)_{n}}q^{\alpha n^2}z^n$$
has infinitely many zeros. Therefore,  the function
$$\sum_{n=0}^{\infty}\frac{h_{n}(x,y|q)}{(q;q)_{n}}q^{\alpha n^2}z^n$$
has infinitely many real zeros and all of them are negative.

We now investigate the function $\sum_{n=0}^{\infty}\frac{g_{n}(x,y|q)}{(q;q)_{n}}q^{\alpha n^2}z^n,$ where $-1<x,y<0$ and $\alpha\geq \frac{1}{2}.$
According to \cite [Theorem 3.1, (3.2)]{L}, we have
\begin{equation*}
\sum_{n=0}^{\infty}(-1)^n \frac{g_{n}(x,y|q)q^{{n\choose 2}}}{(q;q)_{n}}t^n=(x,y;q)_{\infty},
\end{equation*}
which, by Lemma \ref{l1}, implies that
\begin{equation*}
\left\{(-1)^n \frac{g_{n}(x,y|q)q^{{n\choose 2}}}{(q;q)_{n}}\right\}_{n=0}^{\infty}
\end{equation*}
is a PF sequence,
namely,
\begin{equation*}
\left\{(-1)^n \frac{g_{n}(x,y|q)q^{\frac{n^2}{2}}}{(q;q)_{n}}\right\}_{n=0}^{\infty}
\end{equation*}
is a PF sequence.  So, by the fact that $\{\frac{q^{k^2}}{k!}\}_{k=0}^{\infty}$ is a PF sequence and (iii) in Lemma \ref{l2},
\begin{equation*}
\left\{(-1)^n \frac{g_{n}(x,y|q)q^{\alpha n^2}}{(q;q)_{n}}\right\}_{n=0}^{N}
\end{equation*}
is a PF sequence, which, by (i) of Lemma \ref{l2}, means that
\begin{equation*}
  \sum_{n=0}^{N}(-1)^n \frac{g_{n}(x,y|q)q^{\alpha n^2}}{(q;q)_{n}}z^n
\end{equation*}
has only nonpositive zeros. It is obvious that
\begin{equation*}
  \sum_{n=0}^{N}(-1)^n \frac{g_{n}(x,y|q)q^{\alpha n^2}}{(q;q)_{n}}z^n\rightarrow \sum_{n=0}^{\infty}(-1)^n \frac{g_{n}(x,y|q)q^{\alpha n^2}}{(q;q)_{n}}z^n
\end{equation*}
point-wise in $\Omega.$  For $0<q<1,~\alpha> 0$ and each $N\in \mathbb{N},~z\in \mathbb{C},$
\begin{equation*}
  \left|\sum_{n=0}^{\infty}(-1)^n \frac{g_{n}(x,y|q)q^{\alpha n^2}}{(q;q)_{n}}z^n\right|\leq \sum_{n=0}^{\infty}(-1)^n\frac{g_{n}(x,y|q)q^{\alpha n^2}}{(q;q)_{n}}|z|^n<+\infty.
\end{equation*}
Then, by Lemma \ref{l3},
\begin{equation*}
  \sum_{n=0}^{N}(-1)^n \frac{g_{n}(x,y|q)q^{\alpha n^2}}{(q;q)_{n}}z^n\rightarrow \sum_{n=0}^{\infty}(-1)^n \frac{g_{n}(x,y|q)q^{\alpha n^2}}{(q;q)_{n}}z^n
\end{equation*}
uniformly on every compact subset of $\Omega.$ We apply Lemma \ref{l4} to  derive that the function
\begin{equation*}
  \sum_{n=0}^{\infty}(-1)^n \frac{g_{n}(x,y|q)q^{\alpha n^2}}{(q;q)_{n}}z^n\neq 0
\end{equation*}
in $\Omega.$ This shows that
\begin{equation*}
  \sum_{n=0}^{\infty}(-1)^n \frac{g_{n}(x,y|q)q^{\alpha n^2}}{(q;q)_{n}}z^n
\end{equation*}
has no zeros outside the set $\{x+yi|x\in (-\infty, 0], y=0\}.$

It is clear that
\begin{align*}
  \left|(-1)^n \frac{g_{n}(x,y|q)}{(q;q)_{n}}\right|&\leq \sum_{k=0}^{n}\frac{|x|^k}{(q;q)_{k}}\frac{|y|^{n-k}}{(q;q)_{n-k}}\\
  &\leq \sum_{k=0}^{\infty}\frac{|x|^k}{(q;q)_{k}}\sum_{k=0}^{\infty}\frac{|y|^{k}}{(q;q)_{k}}\\
  &=\frac{1}{(|x|,|y|;q)_{\infty}}.
\end{align*}
By Lemma \ref{l5} and \cite [Lemma 14.1.4]{I}, the function
\begin{equation*}
  \sum_{n=0}^{\infty}(-1)^n \frac{g_{n}(x,y|q)q^{\alpha n^2}}{(q;q)_{n}}z^n
\end{equation*}
has infinitely many zeros. Hence,
\begin{equation*}
  \sum_{n=0}^{\infty}(-1)^n \frac{g_{n}(x,y|q)q^{\alpha n^2}}{(q;q)_{n}}z^n
\end{equation*}
has infinitely many real zeros and all of them are negative, namely,
\begin{equation*}
  \sum_{n=0}^{\infty}\frac{g_{n}(x,y|q)q^{\alpha n^2}}{(q;q)_{n}}z^n
\end{equation*}
has infinitely many real zeros and all of them are positive. This finishes the proof of Theorem \ref{tt1}.\qed\\
\emph{Proof of Theorem \ref{t2}.} For $0< q< \frac{1}{2^{\frac{1}{\alpha}}}.$ Put
\begin{equation*}
  A_{k}=\frac{(a_{1})_{k}(a_{2})_{k}\cdots (a_{r})_{k}}{(b_{1})_{k}(b_{2})_{k}\cdots (b_{s})_{k}}q^{\alpha k^2}.
\end{equation*}
Then
\begin{equation*}
\frac{A_{k-1}^2}{A_{k}A_{k-2}}=\prod_{i=1}^{r}\frac{a_{i}+k-2}{a_{i}+k-1}\prod_{j=1}^{s}\frac{b_{i}+k-1}{b_{i}+k-2}q^{-2\alpha}
\end{equation*}
which means that
\begin{equation*}
  \lim_{k\rightarrow \infty}\frac{A_{k-1}^2}{A_{k}A_{k-2}}=q^{-2\alpha}>4.
\end{equation*}
So there exists a positive integer $K_{0}$  such that
\begin{equation*}
  \frac{A_{k-1}^2}{A_{k}A_{k-2}}> 4
\end{equation*}
for $k\geq K_{0}.$ It follows from \cite [Theorem B]{DP} that
\begin{equation*}
\sum_{k=K}^{\infty}\frac{(a_{1})_{k}(a_{2})_{k}\cdots (a_{r})_{k}}{(b_{1})_{k}(b_{2})_{k}\cdots (b_{s})_{k}}q^{\alpha k^2}z^k
\end{equation*}
has only negative zeros for any $K\geq K_{0}.$

On the other hand, by the Stirling's formula and \eqref{e1-1}, we have
\begin{equation*}
\rho\left(\sum_{k=K}^{\infty}\frac{(a_{1})_{k}(a_{2})_{k}\cdots (a_{r})_{k}}{(b_{1})_{k}(b_{2})_{k}\cdots (b_{s})_{k}}q^{\alpha k^2}z^k\right)=0,
\end{equation*}
which, by \cite [Theorem 1.2.5]{I}, implies that the function
\begin{equation*}
  \sum_{k=K}^{\infty}\frac{(a_{1})_{k}(a_{2})_{k}\cdots (a_{r})_{k}}{(b_{1})_{k}(b_{2})_{k}\cdots (b_{s})_{k}}q^{\alpha k^2}z^k
\end{equation*}
has infinitely many zeros. Hence, this function has only infinitely many real zeros and all of them are negative. This concludes the proof of Theorem \ref{t2}.\qed
\vskip 5mm \noindent{\bf Acknowledgement.} The author  would like to thank Prof. Ruiming Zhang for his helpful discussion, comments and suggestions.
This work was  supported by the Initial Foundation for
Scientific Research  of Northwest A\&F University (No. 2452015321).  %I would like to thank the referee for his/her helpful comments.

\end{document}